\documentclass[11pt]{amsart}
\usepackage[margin=1in]{geometry}
\usepackage{amsthm, amsmath,amsfonts,amssymb,euscript,hyperref,graphics,color,slashed,mathrsfs}
\usepackage{graphicx}
\usepackage{comment}
\usepackage{import}
\usepackage{tikz}
\usepackage{latexsym}
\usepackage{mathtools}

\newtheorem{theorem}{Theorem}[section]
\newtheorem{lemma}[theorem]{Lemma}

\newtheorem{corollary}[theorem]{Corollary}
\newtheorem{definition}[theorem]{Definition}

\numberwithin{equation}{section}

\begin{document}
\title {On the largest prime divisor of $n!+1$}

\author{Li LAI}

\address{School of Mathematical Sciences, Fudan University \\ Shanghai, China}
\email{lilaimath@gmail.com}


\begin{abstract}
For an integer $m >1$, we denote by $P(m)$ the largest prime divisor of $m$. We prove that $\limsup_{n \rightarrow +\infty} P(n!+1)/n \geqslant 1+9\log 2>7.238$, which improves a result of Stewart. More generally, for any nonzero polynomial $f(X)$ with integer coefficients, we show that $\limsup_{n \rightarrow +\infty} P(n!+f(n))/n \geqslant 1+9\log2$. This improves a result of Luca and Shparlinski. These improvements come from an additional combinatoric idea to the works mentioned above.
\end{abstract}

\maketitle

\section{Introduction}

Let $n$ be a positive integer, we are interested in the prime divisors of $n!+1$. Does there exist infinitely many positive integers $n$ such that $n!+1$ is a prime number? This is of course wildly open, so we turn our attention to find some lower bounds for the largest prime divisor of $n!+1$.

We denote by $P(m)$ the largest prime divisor of an integer $m > 1$. Clearly, $P(n!+1) > n$. If $p$ is a prime divisor of $n!+1$ for an odd integer $n \geqslant 3$, then Wilson's theorem implies that $p$ also divides $(p-1-n)!+1$. From this, one can easily show that $\limsup_{n \rightarrow +\infty} P(n!+1)/n \geqslant 2$.

As far as I know, the first nontrivial result was obtained by Erd{\"o}s and Stewart \cite{ES76} in 1976. They showed that $\limsup_{n \rightarrow +\infty} P(n!+1)/n > 2$. This problem seems to be forgotten until this century. In 2003, Luca and Shparlinski \cite{LS05a} proved that for any polynomial $f(X) \in \mathbb{Z}[X] \setminus \{ 0 \}$,
\[ \limsup_{n \rightarrow +\infty} \frac{P(n!+f(n))}{n} \geqslant \frac{5}{2}. \]
Later on, Stewart \cite{Ste04} showed that
\[ \limsup_{n \rightarrow +\infty} \frac{P(n!+1)}{n} \geqslant \frac{11}{2}, \]
moreover, for $\varepsilon > 0$, the set of positive integers $n$ for which $P(n!+1) > (11/2-\varepsilon)n$ has positive lower asymptotic density. Recall that for a set $A$ of positive integers, the lower asymptotic density of $A$ is defined by $\liminf_{m \rightarrow +\infty} \#(A \cap \{1,2,\ldots,m\})/m$, where we use $\# S$ to denote the cardinality of a set $S$.

We also mention that in 2002, Murty and Wong \cite{MW02} showed that if $abc$ conjecture is true, then
\[ P(n!+1) > (1+o(1))n\log n.  \]
And in \cite{LS05a}, Luca and Shparlinski showed (unconditionally)
\[ P(n!+1) > n + \left( \frac{1}{4} + o(1) \right)\log n,  \]
which sharpens a result in \cite{ES76}.

In this paper, we improve the constants $5/2$ and $11/2$ in \cite{LS05a,Ste04}.

\begin{theorem}\label{thm1}
Let $f(X) \in \mathbb{Z}[X]\setminus \{ 0 \}$. Then
\[ \limsup_{n \rightarrow +\infty} \frac{P(n!+f(n))}{n} \geqslant 1+9\log2 \approx 7.238. \]
Moreover, for any $\varepsilon > 0$, the set of positive integers $n$ for which $n!+f(n)>1$ and $P(n!+f(n)) >(1+9\log 2 - \varepsilon)n$ has positive lower asymptotic density.
\end{theorem}

We follow most parts of the arguments of Luca and Shparlinski \cite{LS05a} and Stewart \cite{Ste04}. The new ingredient in this paper is Lemma \ref{lem7}. Roughly speaking, for a prime number $p$ and distinct integers $n_1,n_2,\ldots,n_t<p$, we are able to bound from above
\[ \log p \sum_{ct \leqslant t' \leqslant t} \min_{1 \leqslant k \leqslant t'} \operatorname{ord}_{p}(n_{k}!+f(n_k)),  \]
for some constant $c>1/2$. Here $\operatorname{ord}_p(m)$ is the $p$-adic order of an integer $m$. This upper bound gains a multiplicative constant comparing with bounding $(\log p)\min_{1 \leqslant k \leqslant t'} \operatorname{ord}_{p}(n_{k}!+f(n_k))$ individually for each $t'$ as what is done in \cite{LS05a,Ste04}.

Our new ingredient can be applied to some other related problems. For example, in \cite{Ste04}, Stewart proved that
\[ \liminf_{n \rightarrow + \infty \atop n \text{~odd}} \frac{p(n!+1)}{n} \leqslant \frac{\sqrt{145}-1}{8}, \]
where $p(m)$ denotes the least prime divisor of an integer $m > 1$. By inserting our new ingredient, the constant $(\sqrt{145}-1)/8 \approx 1.380$ can be improved to
\[ c = \frac{\sqrt{81(\log 2)^2 + 16} - 9\log 2 + 4}{4} \approx 1.293. \]
(But unlike Theorem \ref{thm1}, our methods cannot prove the lower asymptotic density of the set $\{n: p(n!+1)<(c-\varepsilon_0)n\}$ is positive). In \cite{LS05b}, Luca and Shparlinski showed that
\[ \limsup_{n \rightarrow +\infty} \frac{P(n!+2^n-1)}{n} \geqslant \frac{2\pi^2+3}{18}. \]
The constant $(2\pi^2+3)/18 \approx 1.263$ can be improved to
\[ 1+\frac{2\pi^2-15}{6}\log\frac{3}{2} \approx 1.320. \]

\bigskip

\section{Preliminary lemmas and notations}\label{Sect2}

\begin{lemma}\label{lem1}
Let $f(X) \in \mathbb{Z}[X] \setminus \{ 0 \}$. Then there exists a positive integer $n_0 \geqslant 2$, depending only on $f$, such that the equation
\[ f(n) \prod_{i=1}^{k} (n+i) - f(n+k) = 0 \]
does not have any integer solutions $(n,k)$ with $n \geqslant n_0$ and $k \geqslant 1$.

Moreover, for any integer $n \geqslant n_0$, it holds that $n!+f(n) > 1$ and $f(n) \neq 0$.
\end{lemma}

\begin{proof}
If we delete the last sentence, this is Lemma 3 of \cite{LS05a}. But adding the last sentence is trivial.
\end{proof}

\begin{lemma}\label{lem2}
Let $f(X) \in \mathbb{Z}[X] \setminus \{ 0 \}$. Then there exists a positive constant $C_0$ depending only on $f$, such that for any prime number $p$ and any interval $J \subset [1,p)$ with length $|J| \geqslant 1$, it holds that
\[ \# \left\{ n \in \mathbb{N}:~n \in J,~p \mid (n!+f(n))  \right\} \leqslant C_0|J|^{\frac{2}{3}}. \]
\end{lemma}

\begin{proof}
See Lemma 4 of \cite{LS05a}.
\end{proof}

From now on, we fix a polynomial $f(X) \in \mathbb{Z}[X]\setminus \{ 0 \}$ and let $n_0$ be the constant depending only on $f$ in Lemma \ref{lem1}. We fix an $\varepsilon_0 \in (0,1/100)$.

We use the Vinogradov symbols $\ll$ as well as the Landau symbols $O$ and $o$. Throughout this paper, all the constants implied by $\ll$ or $O$ depend at most on $f$ and $\varepsilon_0$ (in particular, these implied constants do not depend on $x$, a parameter we introduce later). It is more convenient to specify some constants, one of them is $C_1$: we fix a non-negative integer $C_1$, such that
\begin{equation}\label{C1}
|f(n)| \leqslant n^{C_1}, \text{~for any integer~} n \geqslant 2.
\end{equation}

Let $\lambda = 1+9\log 2 - 100\varepsilon_0$. To prove Theorem \ref{thm1}, we will show that the set
\[ B(\lambda) = \{ n \in \mathbb{N}:~ n \geqslant n_0 \text{~and~} P(n!+f(n)) > \lambda n \} \]
has lower asymptotic density $\geqslant \varepsilon_0$.

In fact, we shall show that for all sufficiently large integers $x$, it holds that $|B(\lambda) \cap \{1,2,\ldots,x\}| \geqslant \varepsilon_0 x$. We argue by contradiction, in the following we suppose that $|B(\lambda) \cap \{1,2,\ldots,x\}| < \varepsilon_0 x$. Let
\begin{equation}\label{Z}
Z = \prod_{n=\lceil \varepsilon_0 x \rceil \atop n \notin B(\lambda)}^{x} (n!+f(n)).
\end{equation}
($x$ is large enough so that $\varepsilon_0 x > n_0$. For $n \geqslant \varepsilon_0 x$, we have $n!+f(n)>1$ and $f(n) \neq 0$.) By Stirling's formula, $\log(n!+f(n)) = n\log n +O(n)$, we have
\[ \log Z > \left( \sum_{n = \lceil \varepsilon_0 x \rceil}^{x } \left( n\log n + O(n) \right) \right) - \varepsilon_0 x \left( x\log x + O(x) \right), \]
so for sufficiently large $x$,
\begin{equation}\label{logZ1}
\log Z > \left( \frac{1}{2} - 2\varepsilon_0 \right) x^2 \log x.
\end{equation}

We denote by $|I|$ the length of an interval $I$. The following notation will be frequently used in this paper.

\begin{definition}\label{Def}
Let $n_1,n_2,\ldots,n_t$ be distinct positive integers ($t \geqslant 2$). We define $\mathcal{I}(n_1,n_2,\ldots,n_t)$ to be a set of intervals as follows. Let $\sigma$ be the unique permutation on the index set $\{1,2,\ldots,t\}$ such that $n_{\sigma(1)} < n_{\sigma(2)} < \cdots < n_{\sigma(t)}$, then
\[ \mathcal{I}(n_1,n_2,\ldots,n_t) = \left\{ I:~ I = (n_{\sigma(i)},n_{\sigma(i+1)}] \text{~for some~} i \in \{ 1,2,\ldots,t-1 \} \right\}. \]

We also define
\begin{align*}
&\mathcal{I}_{\text{good}}(n_1,n_2,\ldots,n_t) \\
= &\left\{ I \in \mathcal{I}(n_1,n_2,\ldots,n_t): \text{~either~} |I| \leqslant \frac{x^{0.99}}{t} \text{~or~} I \text{~contains at least $2C_1+1$ prime numbers} \right\}.
\end{align*}

Note that $\mathcal{I}(n_1,n_2,\ldots,n_t)$ is determined by the set $\{ n_1,n_2,\ldots,n_t \}$, but $\mathcal{I}_{\text{good}}(n_1,n_2,\ldots,n_t)$ depends on $t,x$ and the constant $C_1$ from \eqref{C1}.
\end{definition}

Notice that $\# \mathcal{I}(n_1,n_2,\ldots,n_t) = t-1$  and the intervals in $\mathcal{I}(n_1,n_2,\ldots,n_t)$ are disjoint. It is plainly true that
\[ \# \left( \mathcal{I}(n_1,n_2,\ldots,n_t) \setminus \mathcal{I}(n_1,n_2,\ldots,n_{t-1}) \right) \leqslant 2. \]

The following lemma is a variation of the key observation in \cite{Ste04}.

\begin{lemma}\label{lem4}
Suppose that $x$ is larger than some constant depends at most on $f$ and $\varepsilon_0$. Let $t$, $t'$ be integers such that $2 \leqslant t' \leqslant t$. Let $p$ be a prime number and $n_1,n_2,\ldots,n_t$ be distinct positive integers in the interval $[\varepsilon_0 x,\min\{p,x\})$. Then for any two distinct intervals $I_1, I_2 \in \mathcal{I}(n_1,n_2,\ldots,n_{t'}) \cap \mathcal{I}_{\text{good}}(n_1,n_2,\ldots,n_t)$ with $|I_2| \geqslant |I_1|$, we have
\begin{equation}\label{ineq1}
(\log p) \min_{1 \leqslant k \leqslant t'}\{\operatorname{ord}_{p}(n_k! + f(n_k))\} \leqslant  |I_1| \log \left( \frac{ex}{|I_1|} \right) + (|I_2|-|I_1| + \frac{x^{0.99}}{t} + O(1))\log x.
\end{equation}
\end{lemma}

\begin{proof}
There exist indices $i_1,j_1,i_2,j_2 \leqslant t'$ such that $I_1 = (n_{i_1},n_{j_1}], I_2= (n_{i_2},n_{j_2}]$. we denote
\[ D = p^{\min_{1 \leqslant k \leqslant t'}\{\operatorname{ord}_{p}(n_k! + f(n_k))\}}.\]
From $D \mid (n_{i_1}!+f(n_{i_1}))$ and $D \mid (n_{j_1}!+f(n_{j_1}))$ we obtain that
\begin{equation}\label{divi1}
D ~\Bigg|~ \left( f(n_{i_1})\prod_{k=1}^{n_{j_1}-n_{i_1}}(n_{i_1}+k) - f(n_{j_1}) \right).
\end{equation}
Since $n_{j_1}>n_{i_1} \geqslant \varepsilon_0x > n_0$, by Lemma \ref{lem1}, we have
\[ f(n_{i_1})\prod_{k=1}^{n_{j_1}-n_{i_1}}(n_{i_1}+k) - f(n_{j_1}) \neq 0. \]
Hence, we deduce from \eqref{divi1} and \eqref{C1} that
\[ D \leqslant \left| f(n_{i_1})\prod_{k=1}^{n_{j_1}-n_{i_1}}(n_{i_1}+k) - f(n_{j_1}) \right| \leqslant 2x^{C_1+|I_1|}. \]
If $|I_1| \leqslant x^{0.99}/t$, then $\log D \leqslant (x^{0.99}/t + O(1))\log x$, thus \eqref{ineq1} is proved. Similarly, we have
\begin{equation}\label{divi2}
D ~\Bigg|~ \left( f(n_{i_2})\prod_{k=1}^{n_{j_2}-n_{i_2}}(n_{i_2}+k) - f(n_{j_2}) \right) \neq 0.
\end{equation}
If $|I_2| \leqslant x^{0.99}/t$, then \eqref{ineq1} is also proved.

It remains to consider the case that both $|I_1|$ and $|I_2|$ are greater than $x^{0.99}/t$. By the definition of $\mathcal{I}_{\text{good}}(n_1,n_2,\ldots,n_t)$, there exist distinct primes $q_1,q_2,\ldots,q_{2C_1+1}$ in the interval $I_1$ and there exist distinct primes $q_1^{\prime},q_2^{\prime},\ldots,q_{2C_1+1}^{\prime}$ in the interval $I_2$. By \eqref{divi1} and \eqref{divi2}, we have
\begin{equation}\label{divi3}
D ~\Bigg|~ \left( f(n_{i_1})f(n_{j_2})\prod_{k=1}^{n_{j_1}-n_{i_1}}(n_{i_1}+k) -  f(n_{i_2})f(n_{j_1})\prod_{k'=1}^{n_{j_2}-n_{i_2}}(n_{i_2}+k') \right).
\end{equation}
We claim that
\begin{equation}\label{ffP-ffP}
f(n_{i_1})f(n_{j_2})\prod_{k=1}^{n_{j_1}-n_{i_1}}(n_{i_1}+k) -  f(n_{i_2})f(n_{j_1})\prod_{k'=1}^{n_{j_2}-n_{i_2}}(n_{i_2}+k') \neq 0.
\end{equation}
In fact, since $I_1 = (n_{i_1},n_{j_1}]$ and $I_2= (n_{i_2},n_{j_2}]$ are disjoint, there are two cases. The first case is that $n_{i_1}<n_{j_1} \leqslant n_{i_2}<n_{j_2}$. If \eqref{ffP-ffP} does not hold, then we deduce that $q_1^{\prime}q_2^{\prime} \cdots q_{2C_1+1}^{\prime} \mid f(n_{i_1})f(n_{j_2})$. Since $n_{i_1},n_{j_1},n_{i_2},n_{j_2} \geqslant \varepsilon_0 x > n_0$, we have $f(n_{i_1})f(n_{j_2}) \neq 0$. Thus, $q_1^{\prime}q_2^{\prime} \cdots q_{2C_1+1}^{\prime} \leqslant |f(n_{i_1})f(n_{j_2})| \leqslant x^{2C_1}$. But $q_1^{\prime}q_2^{\prime} \cdots q_{2C_1+1}^{\prime} \geqslant (\varepsilon_0 x)^{2C_1+1}$, this is a contradiction because $x$ is large. The second case is that $n_{i_2}<n_{j_2} \leqslant n_{i_1}<n_{j_1}$, if \eqref{ffP-ffP} does not hold, we will get a similar contradiction from $q_1q_2 \cdots q_{2C_1+1} \mid f(n_{i_2})f(n_{j_1})$. Hence, the claim \eqref{ffP-ffP} holds.

Since $|I_2| = n_{j_2} - n_{i_2} \geqslant |I_1| = n_{j_1}-n_{i_1}$, the factorial $(n_{j_1}-n_{i_1})!$ is a divisor of the right hand side of \eqref{divi3}. Since $D$ is a power of the prime $p$, and $n_{i_1},n_{j_1},n_{i_2},n_{j_2} < p$, the factorial $(n_{j_1}-n_{i_1})!$ is coprime to $D$. Thus, we have
\begin{equation}\label{divi4}
D ~\Bigg|~ \frac{1}{(n_{j_1}-n_{i_1})!} \left( f(n_{i_1})f(n_{j_2})\prod_{k=1}^{n_{j_1}-n_{i_1}}(n_{i_1}+k) -  f(n_{i_2})f(n_{j_1})\prod_{k'=1}^{n_{j_2}-n_{i_2}}(n_{i_2}+k') \right).
\end{equation}
Note that $m! \geqslant (m/e)^m$ for any positive integer $m$ and
\[ \left| f(n_{i_1})f(n_{j_2})\prod_{k=1}^{n_{j_1}-n_{i_1}}(n_{i_1}+k) -  f(n_{i_2})f(n_{j_1})\prod_{k'=1}^{n_{j_2}-n_{i_2}}(n_{i_2}+k') \right| \leqslant 2x^{2C_1 + |I_2|}, \]
we deduce from \eqref{divi4} and \eqref{ffP-ffP} that
\[ D \leqslant 2x^{|I_2|-|I_1|+2C_1} \cdot \left( \frac{ex}{|I_1|} \right)^{|I_1|}, \]
which completes the proof of \eqref{ineq1}.
\end{proof}

From the proof of Lemma \ref{lem4} we can see that, we require the intervals in $\mathcal{I}_{\text{good}}$ to contain at least $2C_1+1$ prime numbers only for the purpose to show some quantity is nonzero. The following result of Heath-Brown will help us to show that most intervals in $\mathcal{I}$ are ``good'' (when $t \leqslant C_0x^{2/3}$), and this is the only non-elementary result we need in this paper.

\begin{lemma}[Heath-Brown~\cite{Hea79}]\label{lem5}
Let $2=p_1<p_2<\ldots$ denote the sequence of prime numbers. For any $\varepsilon > 0$, there exists a constant $C_{\varepsilon}$ depending only on $\varepsilon$, such that for any $y \geqslant 2$ we have
\[ \sum_{p_k \leqslant y} (p_{k+1}-p_k)^2 \leqslant C_{\varepsilon} y^{\frac{23}{18}+\varepsilon}. \]
\end{lemma}

\begin{proof}
This is Theorem 1 of \cite{Hea79}.
\end{proof}

\begin{corollary}\label{cor6}
Suppose that $x$ is larger than some constant depends at most on $f$ and $\varepsilon_0$. Let $t$ be an integer such that $2 \leqslant t \leqslant C_0 x^{2/3}$, where $C_0$ is the constant in Lemma \ref{lem2}. Let $n_1,n_2,\ldots,n_t$ be any distinct integers in the interval $[\varepsilon_0 x, x]$. Then we have
\begin{equation}\label{most_are_good}
\# \mathcal{I}_{\text{good}}(n_1,n_2,\ldots,n_t) \geqslant t-2t^{0.99}.
\end{equation}
\end{corollary}

\begin{proof}
Since $x$ is large, we have $t \leqslant C_0x^{\frac{2}{3}} < x^{\frac{2}{3} + 0.01} < x$. By Lemma \ref{lem5}, there exists an absolute constant $C$ such that
\begin{equation}\label{SOS}
 \sum_{p_k \leqslant y} (p_{k+1}-p_k)^2 \leqslant Cy^{\frac{23}{18}+0.001}
\end{equation}
for any $y \geqslant 2$. (Recall we denote by $2=p_1<p_2<\ldots$ the sequence of prime numbers.)

For any interval $I \in \mathcal{I}(n_1,n_2,\ldots,n_t) \setminus \mathcal{I}_{\text{good}}(n_1,n_2,\ldots,n_t)$, we have $|I|> \frac{x^{0.99}}{t}$ and $I$ contains at most $2C_1$ prime numbers. There are indices $i,j$ such that $I = (n_i,n_j]$. Let $p_{k+1}$ be the least prime number such that $p_{k+1} > n_j$. Then $p_k \leqslant n_j \leqslant x$. On the other hand, $p_{k+1} > n_j \geqslant \varepsilon_0 x$, we have $k > 2C_1$ since $x$ is large. Since $I$ contains at most $2C_1$ prime numbers, we must have $p_{k-2C_1} \leqslant n_i$, so $I \subset (p_{k-2C_1},p_{k+1})$.

Let
\[ m = \# \left( \mathcal{I}(n_1,n_2,\ldots,n_t) \setminus \mathcal{I}_{\text{good}}(n_1,n_2,\ldots,n_t) \right). \]
In the above paragraph, we showed that for any interval $I \in \mathcal{I}(n_1,n_2,\ldots,n_t) \setminus \mathcal{I}_{\text{good}}(n_1,n_2,\ldots,n_t)$, there exists a prime number $p_k \leqslant x$ (with $k > 2C_1$) such that $I \subset (p_{k-2C_1},p_{k+1})$. Now, for any index $k$ such that $k>2C_1$ and $p_k \leqslant x$, we denote by $m_k$ the number of intervals in $\mathcal{I}(n_1,n_2,\ldots,n_t) \setminus \mathcal{I}_{\text{good}}(n_1,n_2,\ldots,n_t)$ that are contained in $(p_{k-2C_1},p_{k+1})$. Then
\[ \sum_{p_k \leqslant x \atop k > 2C_1} m_k \geqslant m \]
and
\begin{align}\label{m1}
\sum_{p_k \leqslant x \atop k > 2C_1} (p_{k+1} - p_{k-2C_1})^2 &\geqslant \sum_{p_k \leqslant x \atop k > 2C_1} m_k^2 \left(\frac{x^{0.99}}{t}\right)^2  \notag \\
&\geqslant \frac{x^{1.98}}{t^2} \sum_{p_k \leqslant x \atop k > 2C_1} m_k \notag \\
&\geqslant m \cdot \frac{x^{1.98}}{t^2}.
\end{align}
On the other hand, by \eqref{SOS} and the Cauchy-Schwarz inequality, we have
\begin{align}\label{m2}
&\sum_{p_k \leqslant x \atop k > 2C_1} (p_{k+1} - p_{k-2C_1})^2 \notag \\
\leqslant& (2C_1+1) \sum_{p_k \leqslant x \atop k > 2C_1} \left( (p_{k+1}-p_{k})^2 + (p_{k}-p_{k-1})^2 + \cdots + (p_{k-2C_1+1}-p_{k-2C_1})^2 \right) \notag \\
\leqslant& (2C_1+1)^2 \sum_{p_{k} \leqslant x} (p_{k+1}-p_{k})^2 \notag \\
\leqslant&  (2C_1+1)^2 \cdot C \cdot x^{\frac{23}{18}+0.001} < x^{\frac{23}{18}+0.01}. \quad \text{(Because $x$ is large.)}
\end{align}
Combining \eqref{m1} and \eqref{m2} we obtain that
\[ m \leqslant \frac{t^2}{x^{\frac{13}{18}-0.03}} < \frac{t}{x^{\frac{1}{18}-0.04}} < t^{0.99}. \]
(We used $t < x^{\frac{2}{3}+0.01}$ and $t<x$.)
Therefore,
\[ \# \mathcal{I}_{\text{good}}(n_1,n_2,\ldots,n_t) = t-1-m \geqslant t - 2t^{0.99}. \]
The proof of Corollary \ref{cor6} is complete.
\end{proof}

The following lemma is the key step of this paper.

\begin{lemma}\label{lem7}
Suppose that $x$ is larger than some constant depends at most on $f$ and $\varepsilon_0$. Let $p$ be a prime number. Let $t$ be an integer such that
\[ \left( \frac{10}{\varepsilon_0} \right)^{100} \leqslant t \leqslant C_0 x^{\frac{2}{3}}, \]
where $C_0$ is the constant in Lemma \ref{lem2}. Let $J$ be an interval such that
\[ J \subset [\varepsilon_0x, \min\{x,p\}). \]
If $n_1,n_2,\ldots,n_t$ are distinct integers in the interval $J$, then we have
\begin{equation}\label{ineq2}
\log p \sum_{t' = \left\lceil \left(\frac{1+\varepsilon_0}{2}\right)t \right\rceil}^{t} \min_{1 \leqslant j \leqslant t'} \left\{ \operatorname{ord}_{p} \left( n_{j}! + f(n_j) \right) \right\} \leqslant \frac{|J|}{2} \log t + \frac{x\log x}{t^{0.98}} + O(x).
\end{equation}
\end{lemma}

\begin{proof}
Let $t_1 = \# \mathcal{I}_{\text{good}}(n_1,n_2,\ldots,n_t)$. Plainly, $t_1 \leqslant t-1$. By Corollary \ref{cor6}, we have
\begin{equation}\label{t1}
t_1 \geqslant t-2t^{0.99}.
\end{equation}
We list the lengths of all intervals in $\mathcal{I}_{\text{good}}(n_1,n_2,\ldots,n_t)$ in ascending order:
\begin{equation}\label{gamma}
\gamma_1 \leqslant \gamma_2 \leqslant \cdots \leqslant \gamma_{t_1}.
\end{equation}
Since $n_1,n_2,\ldots,n_t \in J$, and the intervals in $\mathcal{I}_{\text{good}}(n_1,n_2,\ldots,n_t)$ are disjoint, it holds that
\begin{equation}\label{sum_of_gamma}
\sum_{k=1}^{t_1} \gamma_k \leqslant |J|.
\end{equation}
Let
\begin{equation}\label{t2}
t_2 = \lfloor t-3t^{0.99} \rfloor.
\end{equation}
Then from $(t_1 - t_2+1)\gamma_{t_2} \leqslant  \sum_{t'=t_2}^{t_1} \gamma_{t'} \leqslant |J|$ and \eqref{t1}\eqref{t2} we deduce that
\begin{equation}\label{gamma_t2}
\gamma_{t_2} \leqslant \frac{|J|}{t^{0.99}}.
\end{equation}

For any $t' \leqslant t$, from the definition of $\mathcal{I}(n_1,n_2,\ldots,n_{t'})$ (see Definition \ref{Def}) we immediately see that
\[ \# \left( \mathcal{I}(n_1,n_2,\ldots,n_{t'}) \setminus \mathcal{I}(n_1,n_2,\ldots,n_{t'-1})  \right) \leqslant 2. \]
Let $k$ be an integer in the range
\begin{equation}\label{k_range}
0 \leqslant k \leqslant t_3, \quad t_3 = \left\lfloor \frac{1}{2}\left( t - 5t^{0.99} \right) \right\rfloor,
\end{equation}
then we have
\begin{equation}\label{t1-2k}
\# \left( \mathcal{I}(n_1,n_2,\ldots,n_{t-k}) \cap \mathcal{I}_{\text{good}}(n_1,n_2,\ldots,n_t) \right) \geqslant t_1 - 2k.
\end{equation}
In particular, for $k$ in the range \eqref{k_range}, we have (recall \eqref{t1})
\[ \# \left( \mathcal{I}(n_1,n_2,\ldots,n_{t-k}) \cap \mathcal{I}_{\text{good}}(n_1,n_2,\ldots,n_t) \right) \geqslant 3t^{0.99}. \]
Let $I_1^{(k)},I_2^{(k)},\ldots,I_{\lceil t^{0.99}\rceil + 1}^{(k)}$ be the shortest $\lceil t^{0.99}\rceil + 1$ intervals in $\mathcal{I}(n_1,n_2,\ldots,n_{t-k}) \cap \mathcal{I}_{\text{good}}(n_1,n_2,\ldots,n_t)$ such that
\[ |I_1^{(k)}| \leqslant |I_2^{(k)}| \leqslant \cdots \leqslant |I_{\lceil t^{0.99}\rceil + 1}^{(k)}|. \]
Then, by \eqref{t1-2k} and \eqref{gamma} we obtain that
\begin{equation}\label{I_t_0.99_(k)1}
|I_{i}^{(k)}| \leqslant \gamma_{2k+i}, \text{~for~} i=1,2,\ldots,\lceil t^{0.99}\rceil + 1.
\end{equation}
In particular, taking \eqref{k_range} and \eqref{gamma_t2} into consideration, we have
\begin{equation}\label{I_t_0.99_(k)2}
|I_{\lceil t^{0.99}\rceil + 1}^{(k)}| \leqslant \gamma_{t_2} \leqslant \frac{|J|}{t^{0.99}}.
\end{equation}

Let $k$ be in the range \eqref{k_range}. By Lemma \ref{lem4}, for any $i \in \{1,2,\ldots,\lceil t^{0.99}\rceil\}$, we have
\begin{equation}\label{use_lem4}
\log p \min_{1 \leqslant j \leqslant t-k}\{\operatorname{ord}_{p}(n_{j}! + f(n_{j}))\} \leqslant  |I_i^{(k)}| \log \left( \frac{ex}{|I_{i}^{(k)}|} \right) + (|I_{i+1}^{(k)}|-|I_i^{(k)}| + \frac{x^{0.99}}{t} + O(1))\log x.
\end{equation}
Taking the average of \eqref{use_lem4} over $i \in \{1,2,\ldots,\lceil t^{0.99}\rceil\}$, and notice that the function $y \mapsto y\log(ex/y)$ is increasing on $y \in (0,x)$, it follows that
\[ \log p \min_{1 \leqslant j \leqslant t-k}\{\operatorname{ord}_{p}(n_{j}! + f(n_{j}))\} \leqslant |I_{\lceil t^{0.99}\rceil}^{(k)}| \log \left( \frac{ex}{|I_{\lceil t^{0.99}\rceil}^{(k)}|} \right) + \left(  \frac{|I_{\lceil t^{0.99}\rceil + 1}^{(k)}|}{t^{0.99}}  + \frac{x^{0.99}}{t} + O(1) \right) \log x,\]
then, bounding $|I_{\lceil t^{0.99}\rceil}^{(k)}|$ by \eqref{I_t_0.99_(k)1} and bounding $|I_{\lceil t^{0.99}\rceil + 1}|$ by \eqref{I_t_0.99_(k)2}, we obtain that
\begin{equation}\label{t-k}
\log p  \min_{1 \leqslant j \leqslant t-k}\{\operatorname{ord}_{p}(n_{j}! + f(n_{j}))\} \leqslant \gamma_{2k+\lceil t^{0.99}\rceil} \log \left(  \frac{ex}{\gamma_{2k+\lceil t^{0.99}\rceil}} \right) + \left(  \frac{|J|}{t^{1.98}} + \frac{x^{0.99}}{t} + O(1)  \right) \log x.
\end{equation}

Summing \eqref{t-k} over $k$ in the range \eqref{k_range}, we obtain
\begin{align}\label{ineq2_1}
&\log p \sum_{k=0}^{t_3}  \min_{1 \leqslant j \leqslant t-k}\{\operatorname{ord}_{p}(n_{j}! + f(n_{j}))\} \notag \\
&\leqslant \left( \sum_{k=0}^{t_3} \gamma_{2k+\lceil t^{0.99}\rceil} \log \left(  \frac{ex}{\gamma_{2k+\lceil t^{0.99}\rceil}} \right) \right) + \left(  \frac{|J|}{t^{0.98}} + x^{0.99} + O(t)  \right) \log x \notag \\
&\leqslant \left( \sum_{k=0}^{t_3} \gamma_{2k+\lceil t^{0.99}\rceil} \log \left(  \frac{ex}{\gamma_{2k+\lceil t^{0.99}\rceil}} \right) \right) + \frac{x\log x}{t^{0.98}} + O(x),
\end{align}
where in the last inequality we used $|J| < x$ and $t \leqslant C_0 x^{\frac{2}{3}}$.

Finally, since the function $y \mapsto y\log(ex/y)$ is concave and increasing on $y \in (0,x)$, and since
\begin{align*}
\sum_{k=0}^{t_3} \gamma_{2k+\lceil t^{0.99}\rceil} &\leqslant \frac{1}{2} \sum_{k=0}^{t_3} \left( \gamma_{2k+\lceil t^{0.99}\rceil} + \gamma_{2k+\lceil t^{0.99}\rceil+1} \right) \\
&< \frac{1}{2} \sum_{k=1}^{t_1} \gamma_k \\
&\leqslant \frac{|J|}{2},
\end{align*}
we have (by the Jensen's inequality and the monotonicity)
\begin{align}\label{ineq2_2}
\sum_{k=0}^{t_3} \gamma_{2k+\lceil t^{0.99}\rceil} \log \left(  \frac{ex}{\gamma_{2k+\lceil t^{0.99}\rceil}} \right) &< (t_3+1) \cdot \frac{|J|/2}{t_3+1} \log \left(  \frac{ex}{(|J|/2)/(t_3+1)}  \right) \notag \\
&< \frac{|J|}{2} \log \left(  \frac{ext}{|J|/2} \right) = \frac{|J|}{2} \log t + \frac{|J|}{2} \log \left(  \frac{ex}{|J|/2} \right) \notag \\
&<  \frac{|J|}{2} \log t + x,
\end{align}
where in the last inequality we used $|J|/2<x$ and the the function $y \mapsto y\log(ex/y)$ is increasing on $y \in (0,x)$.

By \eqref{ineq2_1}, \eqref{ineq2_2}, and the fact $t-t_3 <  \lceil \left(\frac{1+\varepsilon_0}{2}\right)t \rceil$ (recall $t_3 = \left\lfloor \frac{1}{2}\left( t - 5t^{0.99} \right) \right\rfloor$ and $t > (10/\varepsilon_0)^{100}$), we complete the proof of \eqref{ineq2}.
\end{proof}

\section{Proof of Theorem \ref{thm1}}

Recall the definition of $B(\lambda)$ and $Z$ (see \eqref{Z}), we know that $P(Z) \leqslant \lambda x$.

For a prime number $p$, we denote
\[ N_p = \{ n \in \mathbb{N}:~n \in [\varepsilon_0x,x], n \notin B(\lambda), \text{~and~} p \mid (n!+f(n)) \}. \]
Then,
\begin{equation}\label{Np0}
(\log p)\operatorname{ord}_p(Z) = \log p \sum_{n \in N_p} \operatorname{ord}_p(n!+f(n)).
\end{equation}

If $n \geqslant p$ and $p \mid (n!+f(n))$, then $f(n) \equiv 0 \pmod{p}$. Let $a \neq 0$ be the leading coefficient of $f(X)$. If $p>|a|$, then $f(X) \pmod{p}$ has degree $\deg f$. Hence, for any two real numbers $y_2>y_1 \geqslant p$, we have $\# (N_p \cap [y_1,y_2]) \leqslant ((y_2-y_1)/p + 1)\deg f$. That is,
\begin{equation}\label{N_leq_x_over_p}
\# (N_p \cap [y_1,y_2]) \ll \frac{y_2-y_1}{p} + 1, \quad \text{for any~} y_2 >y_1 \geqslant p.
\end{equation}
If $p \leqslant |a|$, then $\# (N_p \cap [y_1,y_2]) \leqslant y_2-y_1+1 \leqslant |a|(y_2-y_1)/p + 1$, so \eqref{N_leq_x_over_p} also holds.

Recall that $|f(n)| \leqslant n^{C_1}$ for all $n \geqslant 2$. Since $f(n) \neq 0$ for $n \geqslant n_0$, we have $\operatorname{ord}_p(f(n)) \leqslant C_1 \log n/\log p$. We can take a large constant $C_2 > n_0$ depending only on $C_1$ and $n_0$ (so $C_2$ depends only on $f$), such that for any prime $p$ and any integer $n \geqslant C_2p$, it holds that $n/p-1 > C_1 \log n/\log p$. Since $\operatorname{ord}_p(n!) > n/p - 1$, we deduce that for any prime $p$ and any $n \geqslant C_2p (> n_0)$, we have $\operatorname{ord}_p(n!+f(n)) = \operatorname{ord}_p(f(n)) = O(\log n/\log p)$. Therefore, for any prime $p$,
\begin{equation}\label{Np1}
\log p \sum_{n \in N_p \atop n \geqslant C_2p}\operatorname{ord}_p(n!+f(n)) \ll  \left( \frac{x}{p} + 1 \right) \cdot \log x \ll x\log x.
\end{equation}

By \eqref{N_leq_x_over_p}, we have $|N_p \cap [p,C_2p)| = O(1)$. Since $\log (n!+f(n)) = O(x\log x)$ for $n_0 < n \leqslant x$, we have
\begin{equation}\label{Np2}
\log p \sum_{n \in N_p \atop p \leqslant n < C_2p}\operatorname{ord}_p(n!+f(n))\ll 1 \cdot x\log x \ll x\log x.
\end{equation}

For any $n \in N_p \cap [1,p)$, by the definition of $B(\lambda)$, we have $p \leqslant \lambda n$. We denote the interval
\begin{equation}\label{Jp}
J_p = \left[\max\{p/\lambda, \varepsilon_0 x\}, \min\{ p,x \}\right).
\end{equation}
Then,
\[ N_p \cap [1,p) \subset J_p.\]

We claim that
\begin{equation}\label{Np3}
\log p \sum_{n \in N_p \atop n<p}\operatorname{ord}_p(n!+f(n)) \leqslant \frac{1}{9\log (2/(1+\varepsilon_0))}|J_p| \log^2 x + O(x\log x).
\end{equation}
In fact, if $N_p \cap [1,p) = \emptyset$, then there is nothing to prove. In the following, we assume that $T:= \#(N_p \cap [1,p)) \geqslant 1$ and
\[  N_p \cap [1,p) = \{ n_1,n_2,\ldots,n_{T} \}. \]
Relabeling if necessary, we can assume that
\begin{equation}\label{ni}
\operatorname{ord}_{p}(n_1! + f(n_1)) \geqslant \operatorname{ord}_{p}(n_2! + f(n_2)) \geqslant \cdots \geqslant \operatorname{ord}_{p}(n_T! + f(n_T)).
\end{equation}
If $|J_p| \geqslant 1$, then Lemma \ref{lem2} implies that $T \leqslant C_0 |J_p|^{\frac{2}{3}} < C_0 x^{\frac{2}{3}}$. If $|J_p| < 1$, then $T \leqslant  1 < C_0 x^{\frac{2}{3}}$. So it is always true that
\begin{equation}\label{T}
T <  C_0 x^{\frac{2}{3}}.
\end{equation}
Let $k^{*}$ be the least non-negative integer such that
\[ \left( \frac{1+\varepsilon_0}{2} \right)^{k^*} T < \left( \frac{10}{\varepsilon_0} \right)^{100}, \]
and we denote
\[ T_k = \left( \frac{1+\varepsilon_0}{2} \right)^{k} T, \quad k=0,1,\ldots,k^{*}. \]
Note that
\begin{equation}\label{k_star}
k^{*} = \frac{\log T}{\log (2/(1+\varepsilon_0))} + O(1).
\end{equation}
For any $0 \leqslant k < k^{*}$, by Lemma \ref{lem7} (with $t = \lfloor T_k \rfloor$ and $J = J_p$) and \eqref{ni}, we have
\begin{equation}\label{Np31}
\log p \sum_{i = \lfloor T_{k+1} \rfloor + 1}^{\lfloor T_k \rfloor} \operatorname{ord}_{p}(n_i! + f(n_i)) \leqslant \frac{|J_p|}{2}\log T_k + \frac{x\log x}{\lfloor T_k \rfloor^{0.98}} + O(x).
\end{equation}
Since $\log (n!+f(n)) = O(x\log x)$ for $n_0 < n \leqslant x$ and $T_{k^{*}} = O(1)$, we have
\begin{equation}\label{Np32}
\log p \sum_{i=1}^{\lfloor T_{k^{*}} \rfloor} \operatorname{ord}_{p}(n_i! + f(n_i)) = O(x\log x).
\end{equation}
Summing \eqref{Np31} over $k=0,1,\ldots,k^{*}-1$ and \eqref{Np32}, we obtain that
\begin{equation}\label{Np33}
\log p \sum_{i=1}^{T} \operatorname{ord}_{p}(n_i! + f(n_i)) = \frac{|J_p|}{2} \sum_{k=0}^{k^{*}-1} \log T_k + x\log x \sum_{k=0}^{k^{*}-1} \frac{1}{\lfloor T_k \rfloor^{0.98}} + O(x\log x) + O(k^{*}x).
\end{equation}
We estimate each term in the right hand side of \eqref{Np33} as follows. Since
\begin{align*}
\sum_{k=0}^{k^{*}-1} \log T_k &= \sum_{k=0}^{k^{*}-1} \left( \log T - k\log\left( \frac{2}{1+\varepsilon_0} \right) \right) = k^{*}\log T - \frac{(k^{*}-1)k^{*}}{2} \log \left( \frac{2}{1+\varepsilon_0} \right) \\
&= \frac{1}{2\log(2/(1+\varepsilon_0))}\log^2 T + O(\log T) \quad (\text{by \eqref{k_star}}) \\
& \leqslant \frac{2}{9\log(2/(1+\varepsilon_0))} \log^{2} x + O(\log x),  \quad (\text{by \eqref{T}})
\end{align*}
we have
\[ \frac{|J_p|}{2} \sum_{k=0}^{k^{*}-1} \log T_k \leqslant \frac{1}{9\log(2/(1+\varepsilon_0))} |J_p| \log^2 x + O(x\log x). \]
Since $T_k$ decreases exponentially, we have
\[ \sum_{k=0}^{k^{*}-1} \frac{1}{\lfloor T_k \rfloor^{0.98}} \ll  \frac{1}{\lfloor T_{k^{*}-1} \rfloor^{0.98}} \ll 1. \]
By \eqref{k_star} and \eqref{T},
\[ k^{*}x = O(x\log x). \]
Hence, plugging these estimates above into \eqref{Np33}, we deduce that
\[ \log p \sum_{i=1}^{T} \operatorname{ord}_{p}(n_i! + f(n_i)) \leqslant \frac{1}{9\log (2/(1+\varepsilon_0))}|J_p| \log^2 x + O(x\log x), \]
this is exactly \eqref{Np3} we claimed.

Combining \eqref{Np0}, \eqref{Np1}, \eqref{Np2} and \eqref{Np3}, we have
\begin{equation}\label{ineq3}
(\log p) \operatorname{ord}_p(Z) \leqslant \frac{1}{9\log (2/(1+\varepsilon_0))}|J_p| \log^2 x + O(x\log x),
\end{equation}
which holds for any prime number $p$. Since $P(Z) \leqslant \lambda x$ and $|J_p| \leqslant \min\{ p,x \} - p/\lambda$, summing \eqref{ineq3} over primes $p \leqslant \lambda x$ we obtain that
\[\log Z \leqslant  \frac{1}{9\log (2/(1+\varepsilon_0))} \left(  \sum_{p \leqslant x} (1-1/\lambda)p  + \sum_{x<p\leqslant \lambda x} (x - p/\lambda) \right) \log^2 x + O(x^2), \]
where for the error term we used $\pi(\lambda x) \ll x/\log x$. By the fact $\pi(y) = (1+o(1))y/\log y$ and
\[ \sum_{p \leqslant y} p = \left( \frac{1}{2}+o(1) \right) \frac{y^2}{\log y} \quad \text{as~} y \rightarrow +\infty, \]
we deduce that
\begin{equation}\label{logZ2}
\log Z \leqslant \left( \frac{1}{9\log (2/(1+\varepsilon_0))}\frac{\lambda - 1}{2} + o(1) \right) x^2 \log x, \quad \text{as~} x \rightarrow +\infty.
\end{equation}

By comparing \eqref{logZ1} with \eqref{logZ2}, we have
\[ \frac{1}{9\log (2/(1+\varepsilon_0))}\frac{\lambda - 1}{2} \geqslant \frac{1}{2} - 2\varepsilon_0, \]
which implies that
\[ \lambda \geqslant 1+9(1-4\varepsilon_0)\log\left( \frac{2}{1+\varepsilon_0} \right), \]
but this contradicts $\lambda = 1+9\log 2 - 100\varepsilon_0$.

Hence, the above contradiction shows that the lower asymptotic density of $B(\lambda)$ must be greater than or equal to $\varepsilon_0$. The proof of Theorem \ref{thm1} is complete.

\end{document}